\documentclass[a4paper]{amsart}

\usepackage{amsmath,amssymb,amsthm}
\usepackage{fullpage}
\usepackage[colorlinks,linkcolor=blue,urlcolor=blue,citecolor=blue,destlabel=true]{hyperref}
\usepackage{tikz-cd}
\usepackage{mathtools}
\usepackage{array} 
\usepackage{caption}

\theoremstyle{plain}
\newtheorem{lemma}{Lemma}
\newtheorem{theorem}[lemma]{Theorem}
\newtheorem*{theorem-main}{Theorem~\ref{thm:main}}

\theoremstyle{definition}
\newtheorem{definition}[lemma]{Definition}

\newtheorem{question}[lemma]{Question}

\newcommand{\N}{\mathbb{N}}

\newcommand{\R}{\mathbb{R}}

\newcommand{\Z}{\mathbb{Z}}

\newcommand{\vr}[2]{\mathrm{VR}(#1;#2)}

\newcommand{\diam}{\mathrm{diam}}

\newcommand{\redhom}{\widetilde{H}}
\newcommand{\cl}{\mathrm{Cl}}

\newcommand{\lk}{\mathrm{lk}}
\newcommand{\st}{\mathrm{st}}
\newcommand{\susp}{\Sigma\,}
\newcommand{\SC}{\mathrm{SC}}
\newcommand{\kg}{\mathrm{KG}}

\begin{document}

\title{On Vietoris--Rips complexes of hypercube graphs}
\author{Micha{\l} Adamaszek}
\email{michal.adamaszek@mosek.com}
\author{Henry Adams}
\email{henry.adams@colostate.edu}

\begin{abstract}
We describe the homotopy types of Vietoris--Rips complexes of hypercube graphs at small scale parameters.
In more detail, let $Q_n$ be the vertex set of the hypercube graph with $2^n$ vertices, equipped with the shortest path metric.
Equivalently, $Q_n$ is the set of all binary strings of length $n$, equipped with the Hamming distance.
The Vietoris--Rips complex of $Q_n$ at scale parameter zero is $2^n$ points, and the Vietoris--Rips complex of $Q_n$ at scale parameter one is the hypercube graph, which is homotopy equivalent to a wedge sum of circles.
We show that the Vietoris--Rips complex of $Q_n$ at scale parameter two is homotopy equivalent to a wedge sum of 3-spheres, and furthermore we provide a formula for the number of 3-spheres.
Many questions about the Vietoris--Rips complexes of $Q_n$ at larger scale parameters remain open.
\end{abstract}



\maketitle

\section{Introduction}

It is an important question in applied topology to understand the Vietoris--Rips complexes of a variety of shapes, such as finite metric spaces and manifolds.
Indeed, one of the most frequently computed forms of persistent homology is obtained by building a Vietoris--Rips simplicial complex filtration on top of a potentially high-dimensional dataset~\cite{Carlsson2009,bauer2021ripser}.
Though these complexes are being constructed by scientists from a wide variety of different disciplines (see for example~\cite{bendich2016persistent,CarlssonIshkhanovDeSilvaZomorodian2008,chung2009persistence,de2007coverage,martin2010topology,topaz2015topological,xia2014persistent}), we nevertheless have a limited mathematical understanding of how the homotopy types of Vietoris--Rips complexes behave.
In this paper we consider Vietoris--Rips complexes of hypercube graphs.
These questions on hypercubes arose from work by Kevin Emmett, Ra\'{u}l Rabad\'{a}n, and Daniel Rosenbloom at Columbia University related to the persistent homology formed from genetic trees, reticulate evolution, and medial recombination~\cite{emmett2015quantifying,emmett2016topology}; see also~\cite{camara2016topological,chan2013topology}.

Vietoris--Rips complexes were independently invented by Vietoris in order to define a (co)homology theory for metric spaces~\cite{Vietoris27}, and by Rips for use in geometric group theory~\cite{Gromov}.
In both settings, one approximates a metric space by a Vietoris--Rips simplicial complex at a chosen scale.
Whereas Vietoris took a limit as the scale parameter approached zero, Rips instead often required the scale parameter to be sufficiently large, for example large enough to fill in holes in a $\delta$-hyperbolic group that are not essential up to quasi-isometry.
Applications of Vietoris--Rips complexes in computational topology~\cite{EdelsbrunnerHarer}, say to recover the ``shape" of a finite dataset~\cite{Carlsson2009}, instead take the scale parameter to be in an intermediate range: small depending on the curvature of the underlying space from which the data was sampled, but not tending to zero.
The theory behind these applications was initiated in work by Hausmann~\cite{Hausmann1995} and Latschev~\cite{Latschev2001}, and advanced in work by Chazal et al.~\cite{chazal2009gromov,ChazalOudot2008,ChazalDeSilvaOudot2014}.

There has been a lot of recent activity on advancing the theory of Vietoris--Rips complexes, including the study of Vietoris--Rips complexes of planar points~\cite{chambers2010vietoris,AFV}, the circle~\cite{Adamaszek2013,AA-VRS1,AAFPP-J},
metric graphs~\cite{gasparovic2018complete}, geodesic spaces~\cite{virk20201,virk2017approximations}, wedge sums~\cite{Adamaszek2020,lesnick2020quantifying}, metric subspaces~\cite{virk2021footprints}, and equivariant spaces~\cite{AdamsHeimPeterson,varisco2021equivariant}.
Vietoris--Rips complexes of higher-dimensional spheres have been studied through neighborhoods of the Kuratowski embedding~\cite{lim2020vietoris} and through Vietoris--Rips metric thickenings~\cite{AAF}, which have recently been shown to have the same persistent homology as Vietoris--Rips simplicial complexes~\cite{AMMW,MoyMasters}.
Furthermore, Hausmann's question about the homotopy connectivity of Vietoris--Rips complexes has been answered in the negative~\cite{virk2021counter}.
Vietoris--Rips complexes have recently been related to Bestvina--Brady Morse theory~\cite{zaremsky2019},
to Borsuk--Ulam theorems mapping into higher-dimensional codomains~\cite{adams2019borsuk,ABF2}, to Dowker's theorem and nerve lemmas~\cite{virk2021rips},
and to the filling radius~\cite{lim2020vietoris,okutan2019persistence} as studied by Gromov~\cite{gromov1983filling,gromov2007metric}
and Katz~\cite{katz1983filling,katz1989diameter,katz9filling,katz1991neighborhoods}.
The study of Vietoris--Rips complexes is finding connections to many different subareas of geometry and topology, including combinatorial topology, metric geometry, equivariant topology, polytope theory, and quantitative topology,
among others.
Our paper directly relates to the study of independence complexes of Kneser graphs~\cite{barmak2013star}, and to K\"{u}nneth formulas for persistent homology~\cite{carlsson2020persistent,gakhar2019k};
we hope that the study of Vietoris--Rips complexes of hypercube graphs will find further connections to other areas of mathematics.

In this paper we study the homotopy types of Vietoris--Rips complexes of hypercube graphs at small scale parameters.
Let $Q_n$ be the vertex set of the hypercube graph with $2^n$ vertices, equipped with the shortest path metric; see Figure~\ref{fig:cubes}.
The metric space $Q_n$ can also be described the set of all binary strings of length $n$, equipped with the Hamming distance.
The Vietoris--Rips complex $\vr{Q_n}{r}$ at scale parameter $r\ge 0$ is the simplicial complex with vertex set $Q_n$, in which $\sigma\subseteq Q_n$ is included as a simplex if the diameter of $\sigma$ is at most $r$.
We see that $\vr{Q_n}{0}$ is $2^n$ disjoint vertices, and that $\vr{Q_n}{1}$ is the hypercube graph, which is homotopy equivalent to the $((n-2)2^{n-1}+1)$-fold wedge sum of circles.
We show that $\vr{Q_n}{2}$ is homotopy equivalent to a wedge sum of 3-spheres.

\begin{figure}[htb]
\centering
\includegraphics[width=\textwidth]{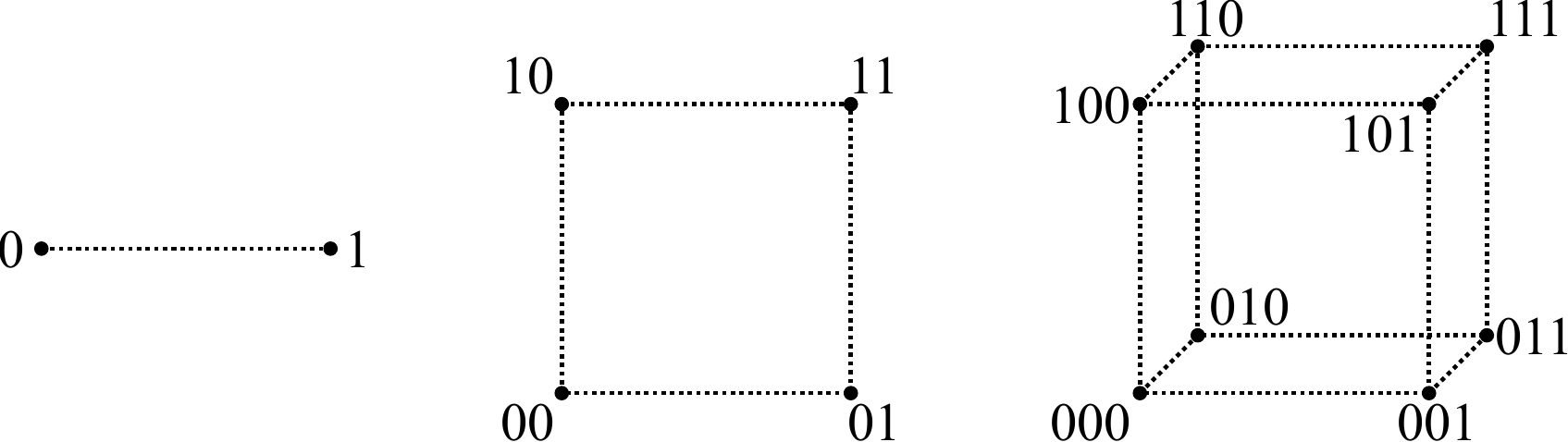}
\caption{
Hypercube graphs $Q_1$, $Q_2$, and $Q_3$.
}
\label{fig:cubes}
\end{figure}

\begin{table}
\def\arraystretch{1.2}
\begin{tabular}{| >{$} c <{$} | >{$} c <{$} | >{$} c <{$} | >{$} c <{$} | >{$} c <{$} | >{$} c <{$} | >{$} c <{$} | >{$} c <{$} | >{$} c <{$} | >{$} c <{$} |}
\hline
& n=1 & 2 & 3 & 4 & 5 & 6 & 7 & 8 & 9 \\
\hline
r=0 & S^0 & \bigvee^3 S^0 & \bigvee^7 S^0 & \bigvee^{15} S^0 & \bigvee^{31} S^0 & \bigvee^{63} S^0 & \bigvee^{127} S^0 & \bigvee^{255} S^0 & \bigvee^{511} S^0 \\
\hline
1 & * & S^1 & \bigvee^5 S^1 & \bigvee^{17} S^1 & \bigvee^{49} S^1 & \bigvee^{129} S^1 & \bigvee^{321} S^1 & \bigvee^{769} S^1 & \bigvee^{1793} S^1 \\
\hline
2 & * & * & S^3 & \bigvee^9 S^3 & \bigvee^{49} S^3 & \bigvee^{209} S^3 & \bigvee^{769} S^3 & \bigvee^{2561} S^3 & \bigvee^{7937} S^3 \\
\hline
3 & * & * & * & S^7 & 
& 
& 
& & \\
\hline
4 & * & * & * & * & S^{15} & & & & \\
\hline
5 & * & * & * & * & * & S^{31} & & & \\
\hline
6 & * & * & * & * & * & * & S^{63} & & \\
\hline
7 & * & * & * & * & * & * & * & S^{127} & \\
\hline
8 & * & * & * & * & * & * & * & * & S^{255} \\
\hline
\end{tabular}
\caption{Homotopy types of $\vr{Q_n}{r}$.}
\label{table:homotopy-types}
\end{table}

\begin{theorem}\label{thm:main}
Let $n\ge 3$.
The Vietoris--Rips complex $\vr{Q_n}{2}$ is homotopy equivalent to the $c_n$-fold wedge sum of 3-spheres, namely 
\[\vr{Q_n}{2}\simeq\textstyle{\bigvee^{c_n}}S^3,\quad\quad\text{where}\quad\quad c_n=\sum_{0\le j<i<n}(j+1)(2^{n-2}-2^{i-1}).\]
\end{theorem}

The first twelve values of $c_n$, starting at $n=3$, are \[1,\ 9,\ 49,\ 209,\ 769,\ 2561,\ 7937,\ 23297,\ 65537,\ 178177,\ 471041,\ 1216513,\ \ldots.\]
We note that $c_n=A027608(n-3)$ for all $n\ge 3$, where the sequence A027608 is described on the Online Encyclopedia of Integer Sequences (OEIS)~\cite{sloane2003line}.
This sequence corresponds to the coefficients in the Taylor series expansion of $\frac{1}{(1-x)(1-2x)^4}$;
see~\cite{albert2012enumeration,crane2015left,hinz2013tower} for other appearances of this sequence.

Table~\ref{table:homotopy-types} gives some of the known homotopy types of $\vr{Q_n}{r}$.
When $r=n-1$, we observe that $\vr{Q_n}{n-1}$ is the boundary of the $(2^{n-1})$-dimensional cross-polytope with $2^n$ vertices, and therefore $\vr{Q_n}{n-1}$ is homeomorphic to the $(2^{n-1}-1)$-dimensional sphere.
For example, when $n=3$, the 1-skeleton of $\vr{Q_3}{2}$ contains all possible edges, except for the edges between the ``antipodal'' vertex pairs $000$ and $111$, $001$ and $110$, $010$ and $101$, and $011$ and $100$ that are at a distance of $n=3$ apart.
Therefore the clique complex is the boundary of the $4$-dimensional cross-polytope with $8$ vertices.
This is the easiest way to see that $\vr{Q_3}{2}=S^3$.
Since $Q_n$ has diameter $n$, for $r\ge n$ the Vietoris--Rips complex $\vr{Q_n}{r}$ is a simplex, and therefore is contractible.

A better understanding of the homotopy types of $\vr{Q_n}{r}$ could relate to stronger versions of the K\"{u}nneth formula for persistent homology of Vietoris--Rips complexes.
Indeed,~\cite{carlsson2020persistent} considers a K\"{u}nneth formula for persistent homology in which the metric on the product $X\times Y$ is given by the sum $d_X+d_Y$.
Motivated by this K\"{u}nneth formula,~\cite[Proposition~4.10]{carlsson2020persistent} gives the persistent homology of $\vr{Q_n}{r}$ for homology up to dimension two.
The 2-dimensional homology is always zero, and all 1-dimensional homology appears prior to scale parameter $r<2$.
Our Theorem~\ref{thm:main} furthermore describes the 3-dimensional homology of $\vr{Q_n}{r}$ that appears at scale $r=2$.
See~\cite{gakhar2019k,lim2020vietoris} and~\cite[Proposition~10.2]{AA-VRS1} for versions of the K\"{u}nneth formula for Vietoris--Rips complexes which hold when the metric on $X\times Y$ is instead the supremum metric.

Many questions about the Vietoris--Rips complexes of $Q_n$ at larger scale parameters remain open.
The topology when $r=3$ is more complicated than the case of $r=2$, in the sense that nonzero reduced homology now appears in more than one dimension.
Indeed, computational evidence shows that various homology groups of $\vr{Q_n}{3}$ are as follows.
The cases of $n=5, 6, 7$ (with integer coefficients) were computed in polymake~\cite{polymake:2000}, and the cases of $n=8, 9$ (with $\Z/2\Z$ coefficients) were computed by Simon Zhang using Ripser++~\cite{zhang2020gpu}, which is a GPU-accelerated version of Ripser~\cite{bauer2021ripser}.

\begin{align*}
&H_i(\vr{Q_5}{3};\Z) &&\cong \Z \text{ for }i=4, &&\cong \Z^{10}\text{ for }i=7, &&\cong 0\text{ for }1\le i\le 7\text{ with }i\neq 4,7\\
&H_i(\vr{Q_6}{3};\Z) &&\cong \Z^{11} \text{ for }i=4, &&\cong \Z^{60}\text{ for }i=7, &&\cong 0\text{ for }1\le i\le 7\text{ with }i\neq 4,7\\
&H_i(\vr{Q_7}{3};\Z) &&\cong \Z^{71} \text{ for }i=4, &&\cong \Z^{280}\text{ for }i=7, &&\cong 0\text{ for }1\le i\le 7\text{ with }i\neq 4,7\\
&H_i(\vr{Q_8}{3};\tfrac{\Z}{2\Z}) &&\cong (\tfrac{\Z}{2\Z})^{351} \text{ for }i=4, &&\cong (\tfrac{\Z}{2\Z})^{1120}\text{ for }i=7, &&\cong 0\text{ for }1\le i\le 9\text{ with }i\neq 4,7\\
&H_i(\vr{Q_9}{3};\tfrac{\Z}{2\Z}) &&\cong (\tfrac{\Z}{2\Z})^{1471} \text{ for }i=4, &&\cong (\tfrac{\Z}{2\Z})^{4032}\text{ for }i=7, &&\cong 0\text{ for }1\le i\le 9\text{ with }i\neq 4,7\\
\end{align*}
The sequence 1, 10, 60, 280, 1120, 4032, \ldots might lead one to conjecture that $H_7(\vr{Q_n}{3};\Z)\cong \Z^{2^{n-4}\binom{n}{4}}$ for all $n\ge 4$.
Similarly, the sequence 1, 11, 71, 351, 1471, \ldots might lead one to conjecture that $H_4(\vr{Q_n}{3};\Z)\cong \Z^{\sum_{i=4}^{n-1}2^{i-4}\binom{i}{4}}$ for all $n\ge 5$.
Interestingly, $2^{n-4}\binom{n}{4}$ is the number of 4-dimensional cube subgraphs in the $n$-dimensional hypercube graph.
Also, the 4-dimensional skeleton of the $n$-dimensional hypercube polytope is homotopy equivalent to a $\left(\sum_{i=4}^{n-1}2^{i-4}\binom{i}{4}\right)$-fold wedge sum of 4-spheres.

\begin{question}\label{ques:wedge}
Are the Vietoris--Rips complexes $\vr{Q_n}{r}$ always homotopy equivalent to a wedge sum of spheres?
\end{question}

\begin{question}\label{ques:homology}
In what dimensions do the Vietoris--Rips complexes $\vr{Q_n}{r}$ have nontrivial reduced homology for $3\le r\le n-2$?
\end{question}

\begin{question}\label{ques:homotopy}
What are the homotopy types of the Vietoris--Rips complexes $\vr{Q_n}{r}$ for $3\le r\le n-2$?
\end{question}

We overview notation and preliminaries in Section~\ref{sec:notation}, prove our main result (Theorem~\ref{thm:main}) in Section~\ref{sec:results}, and conclude by listing some more refined open questions in Section~\ref{sec:conclusion}.

\section{Notation and preliminary results}\label{sec:notation}

We refer the reader to~\cite{Hatcher} for more background on topology.

\subsection*{Topological spaces}

When two topological spaces $X$ and $Y$ are homeomorphic, we write $X=Y$.
We write $X\simeq Y$ to denote that they are homotopy equivalent.
The main results of this paper will show that certain topological spaces are homotopy equivalent to a wedge sum of spheres.
We let $S^n$ denote the $n$-dimensional sphere $S^n=\{x\in\R^{n+1}~|~\|x\|=1\}$.

Given two topological spaces $X$ and $Y$, we denote their wedge sum by $X \vee Y$;
this is the space obtained by gluing $X$ and $Y$ together at a single point.
If $X$ and $Y$ are connected CW complexes, then the homotopy type of $X\vee Y$ is independent of the points chosen to glue together.
For $n\ge 0$ we let $\bigvee^n X$ denote the $n$-fold wedge sum of $X$, namely $\bigvee^n X = X\vee X\vee \ldots \vee X$. 
By convention we let the 0-fold wedge sum of any space $X$ be a single point.

Let $X$ be a topological space.
We let $C(X)=(X\times[0,1])/(X\times \{0\})$ denote the cone over $X$, which is a contractible space.
The suspension of $X$ is formed by joining two cones along their boundary $X$.
We denote the suspension of $X$ by $\susp X = (X \times [0,1]) / \sim$, where $(x,0)\sim (x',0)$ and $(x,1)\sim (x',1)$ for all $x,x'\in X$.
For spheres we have $\susp S^n = S^{n+1}$.

Let $X$ and $Y$ be topological spaces, let $A\subseteq X$ be a subspace, and let $f\colon A\to Y$ be a continuous map.
The adjunction space $Y\cup_f X=(Y\coprod X)/\sim$ is formed by taking a quotient of the disjoint union, where the equivalence relation is generated by $a \sim f(a)$ for all $a\in A$.
Adjunction spaces satisfy convenient homotopy invariance properties.
Indeed, if the inclusion $A\hookrightarrow X$ is a cofibration (for example an inclusion of CW complexes), and if the maps $f,g\colon A\to Y$ are homotopy equivalent, then we have a homotopy equivalence of adjunction spaces $Y\cup_f X\simeq Y\cup_g X$.
Indeed, this is described in~\cite[7.5.5 (Corollary~1)]{brown2006topology} or~\cite[Proposition~5.3.3]{tom2008algebraic}.

\subsection*{Simplicial complexes}

A simplicial complex $K$ on a vertex set $V$ is a family of subsets of $V$, including all singletons, such that if $\sigma\in K$ and $\tau \subseteq \sigma$, then also $\tau\in K$.
We identify a simplicial complex with its geometric realization, which is a topological space.
The \emph{star} of a vertex $v\in C$ is $\st_K(v)=\{\sigma\in K~|~\sigma \cup \{v\}\in K\}$.
Note that the star is contractible since it is a cone with apex $v$.
The \emph{link} of a vertex $v$ is $\lk_K(v)=\{\sigma\in K~|~v\notin \sigma\text{ and }\sigma \cup \{v\}\in K\}$.
For $v\in V$ a vertex, we let $K\setminus v$ denote the induced simplicial complex on vertex set $V\setminus v$; the simplices of $K\setminus v$ are all those simplices $\sigma\in K$ such that $v\notin \sigma$.

\subsection*{A splitting up to homotopy type}

The following preliminary result is likely well-known to combinatorial topologists.

\begin{lemma}\label{lem:splitting}
Let $K$ be a simplicial complex, and let $v\in K$ be a vertex such that the inclusion $\iota\colon \lk_K(v)\hookrightarrow K\setminus v$ is a null-homotopy.
Then up to homotopy we have a splitting $K\simeq (K\setminus v) \vee \Sigma\ \lk_K(v)$.
\end{lemma}

\begin{proof}
By definition of a null-homotopy, the inclusion $\iota\colon \lk_K(v)\hookrightarrow K\setminus v$ is homotopic to a constant map $c\colon \lk_K(v)\to K\setminus v$ that maps to a single point $p\in K\setminus v$.
By the homotopy invariance properties of adjunction spaces, we have the following homotopy equivalence:
\[ K = (K\setminus v)\cup_\iota\st_K(v) \simeq (K\setminus v)\cup_c \st_K(v) = (K\setminus v)\vee \bigl(\{p\}\cup_c\st_K(v)\bigr). \]
In the right-most space, the wedge sum is taken over the point $p$, and the adjunction space $\{p\}\cup_c \st_K(v)$ is arising from the map $c\colon \lk_K(v)\to \{p\}$.
We have 
\[ \{p\}\cup_c\st_K(v) = \st_K(v) / \lk_K(v) \simeq \st_K(v) \cup C(\lk_K(v)) = \susp\lk_K(v), \]
where the homotopy equivalence is described, for example, in~\cite[Example~0.13]{Hatcher}.
Recall $C(\cdot)$ denotes the cone over a topological space.
The last equality is since $\st_K(v)$ is also a cone over $\lk_K(v)$, with apex vertex $v$.
It therefore follows that
\[ K \simeq (K\setminus v)\vee \bigl(\{p\}\cup_c\st_K(v)\bigr) \simeq (K\setminus v)\vee \susp\lk_K(v). \]
\end{proof}

\subsection*{Graphs}

In this paper we consider only simple graphs, i.e.\ graphs in which the edges are undirected, there are no multiple edges, and there are no loop edges.
Let $G$ be a simple graph with vertex set $V$.
The set $E$ of edges is a collection of subsets of $V$ of size two.

We let $N_G(v)=\{u\in V~|~uv\in E\}$ denote the (open) \emph{neighborhood} of vertex $v$ in graph $G$.
That is, $N_G(v)$ contains all vertices not equal to $v$ that are connected to $v$ by an edge.

For $V'\subseteq V$, the subgraph of $G$ induced by the vertex subset $V'$ has $V'$ as its vertex set, and as its edges all edges of $G$ whose vertices are each elements of $V'$.

We let $\cl(G)$ denote the \emph{clique complex} of $G$.
This simplicial complex is the maximal simplicial complex with 1-skeleton equal to $G$.
That is, the vertex set is $V$, and $\cl(G)$ contains a finite set $\sigma\subseteq V$ as a simplex if for any two vertices $u,v\in \sigma$, we have that $uv$ is an edge in $G$.

If $G$ is a graph with vertex set $V$ and edge set $E$, then its \emph{complement} graph $\overline{G}$ also has vertex set $V$, and $\overline{G}$ has an edge $uv$ if and only if $uv\notin E$.
The \emph{independence complex} $I(G)$ of a graph $G$, which is a well-studied simplicial complex in combinatorial topology, can be defined as $I(G)=\cl(\overline{G})$.
Therefore, the study of clique complexes of graphs is closely related to the study of independence complexes.

\subsection*{Vietoris--Rips complexes}

Let $X$ be a metric space, equipped with a distance function $d\colon X\times X\to \R$.
Given a subset $\sigma\subseteq X$, we define its \emph{diameter} to be $\diam(X)=\sup_{x,x'\in\sigma}d(x,x')$.

\begin{definition}
The \emph{Vietoris--Rips complex of $X$ at scale $r$}, denoted $\vr{X}{r}$, is the simplicial complex with vertex set $X$ that contains a finite simplex $\sigma\subseteq X$ if and only if $\diam(\sigma)\le r$.
\end{definition}

Note that a Vietoris--Rips complex is the maximal simplicial complex that can be built on top of its 1-skeleton, and hence it is a \emph{clique} or a \emph{flag} simplicial complex.

We remark that this definition makes sense for any symmetric function $d\colon X\times X\to \R$, even if it is not a metric (for example, even if it does not satisfy the triangle inequality).

\subsection*{Star clusters and Kneser graphs}

We now review star clusters, introduced by Barmak in~\cite{barmak2013star} for studying clique and independence complexes of graphs.
For $K$ a simplicial complex and for $\sigma\in K$ a simplex, the \emph{star cluster} of $\sigma$ is the simplicial complex $\SC_K(\sigma)=\cup_{v\in \sigma}\st_K(v)$; see Figure~\ref{fig:starClusters}.
Lemma~3.2 of \cite{barmak2013star} states that if $K$ is a clique simplicial complex, then $\SC_K(\sigma)$ is contractible for all simplices $\sigma\in K$.
We will simplify notation by writing $\SC(\sigma)$ when $K$ is clear from context.

\begin{figure}[htb]
\centering
\def\svgwidth{4in}
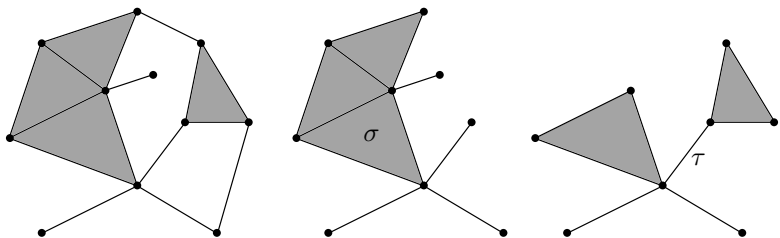
\caption{(Left) A clique simplicial complex, (middle) the star cluster of a 2-simplex $\sigma$, and (right) the star cluster of a 1-simplex $\tau$.}
\label{fig:starClusters}
\end{figure}

The Kneser graph $\kg_{n,k}$ has as its vertices the $k$-element subsets from a set of size $n$, where two vertices are adjacent if their corresponding subsets are disjoint.
Kneser conjectured in 1955~\cite{kneser1955aufgabe} that for $n\ge 2k\ge 2$, the chromatic number of $\kg_{n,k}$ (i.e.\ the number of vertex colors required so that adjacent vertices receive different colors) is $n-2k+2$.
This was first proven in 1978 by Lov{\'a}sz~\cite{Lovasz1978}; the proof uses topological techniques and helped initiate the field of topological combinatorics.
In~\cite[Theorem~4.11]{barmak2013star}, Barmak uses star clusters to prove that for $n\ge 4$, the independence complex of the Kneser graph $I(\kg_{n,2})$ is homotopy equivalent to a wedge sum of ${n-1 \choose 3}$ copies of the 2-sphere.
We use the ideas behind this proof in our verification of Lemma~\ref{lem:link} (in Section~\ref{sec:results}).

\subsection*{Hypercube graphs}

Let $Q_n$ be the set of all $2^n$ binary strings of length $n$.
For example, $Q_3$ consists of eight strings: 000, 001, 010, 011, 100, 101, 110, and 111.
We equip the set $Q_n$ with the Hamming distance: the distance between two binary strings is the number of positions in which their entries differ.
Equivalently, $Q_n$ is the vertex set of the $n$-dimensional hypercube graph (with $2^n$ vertices), equipped with the shortest path distance in the hypercube graph; see Figure~\ref{fig:cubes}.
Alternatively, $Q_n$ is the set of lattice points in $\R^n$ with all coordinates equal to $0$ or $1$, equipped with the $L^1$ or taxicab metric.

\section{Results}\label{sec:results}

Let $n$ be a positive integer.
The Vietoris--Rips complex of $Q_n$ at scale parameter 0 is simply the simplicial complex with $2^n$ vertices and no higher-dimensional simplices.
Therefore $\vr{Q_n}{0}$ is the union of $2^n$ distinct vertices, or alternatively, $\vr{Q_n}{0}=\bigvee^{2^n-1}S^0$ is the $(2^n-1)$-fold wedge sum of 0-spheres.

The Vietoris--Rips complex of $Q_n$ at scale parameter 1 is the underlying hypercube graph.
This is a connected graph with $2^n$ vertices and $n2^{n-1}$ edges, which therefore is homotopy equivalent to the wedge sum of $n2^{n-1}-2^n+1=(n-2)2^{n-1}+1$ circles.
So $\vr{Q_n}{1}\simeq\bigvee^{(n-2)2^{n-1}+1}S^1$ is homotopy equivalent to the $((n-2)2^{n-1}+1)$-fold wedge sum of circles.

The Vietoris--Rips complex of $Q_n$ at scale parameter $n-1$ is is the boundary of the $(2^{n-1})$-dimensional cross-polytope with $2^n$ vertices.
To see this, note that the 1-skeleton of $\vr{Q_n}{n-1}$ consists of all possible edges, except that it is missing the ``antipodal" edges between vertices (such as $00\ldots0$ and $11\ldots1$) that are at the maximal distance $n$ apart.
Taking the clique complex gives that $\vr{Q_n}{n-1}$ is the boundary of the $(2^{n-1})$-dimensional cross-polytope with $2^n$ vertices, and therefore $\vr{Q_n}{n-1}$ is homeomorphic to the $(2^{n-1}-1)$-dimensional sphere.

Our main result is to characterize the homotopy type of the Vietoris--Rips complex of the hypercube graphs at scale parameter 2, which for convenience we restate below.
It is easy to see that $\vr{Q_1}{2}$ and $\vr{Q_2}{2}$ are contractible.

\begin{theorem-main}
Let $n\ge 3$.
The Vietoris--Rips complex $\vr{Q_n}{2}$ is homotopy equivalent to the $c_n$-fold wedge sum of 3-spheres, namely 
\[\vr{Q_n}{2}\simeq\textstyle{\bigvee^{c_n}}S^3,\quad\quad\text{where}\quad\quad c_n=\displaystyle{\sum_{0\le j<i<n}(j+1)(2^{n-2}-2^{i-1})}.\]
\end{theorem-main}

We now build notation to be used in the proof of Theorem~\ref{thm:main}.
Recall that the metric space $Q_n$ is the set of all $2^n$ binary strings of length $n$, namely the numbers from $0$ to $2^n-1$ written in binary, equipped with the Hamming distance.
In order to enable proofs by induction, we consider the metric spaces $V_m$, consisting of all numbers from $0$ to $m-1$ written as binary strings, also equipped with the Hamming distance.
Note that $V_{2^n}=Q_n$.
To ease notation in the following proofs, we also switch to clique complex notation.
Let $G_m^r $ be the graph whose vertex set is $V_m$, and that has an edge between two vertices if their Hamming distance is at most $r$; see Figures~\ref{fig:Gm1} and~\ref{fig:Gm2}.
Note that $\cl(G_m^r)=\vr{V_m}{r}$, and so in particular $\cl(G_{2^n}^r)=\vr{Q_n}{r}$.
In Theorem~\ref{thm:Gm2} we give the homotopy types of $\cl(G_m^2)$ for all $m$, and hence Theorem~\ref{thm:main} will follow as a corollary after letting $m=2^n$.

\begin{figure}[h]
\includegraphics[width=\textwidth]{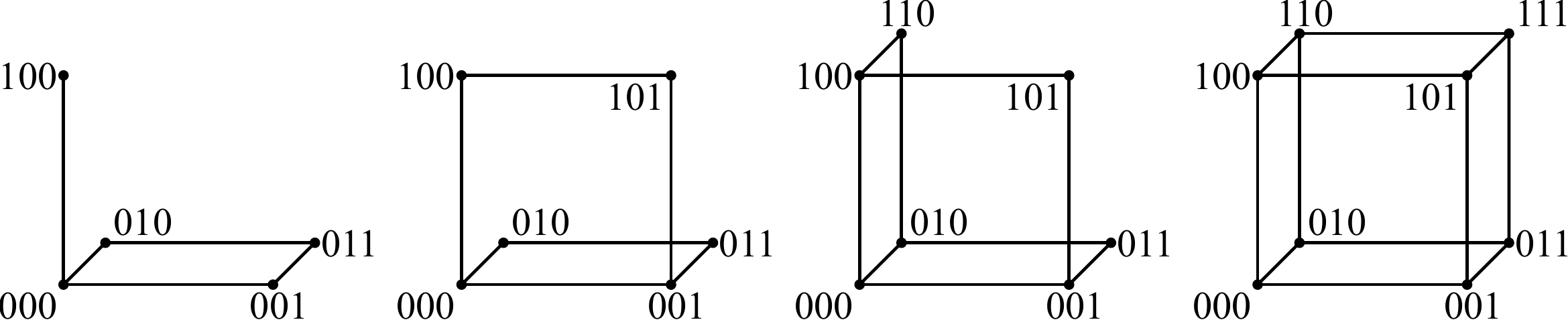}
\caption{
The graphs $G_m^1$ for $m=5,6,7,8$.
}
\label{fig:Gm1}
\end{figure}

\begin{figure}[h]
\includegraphics[width=\textwidth]{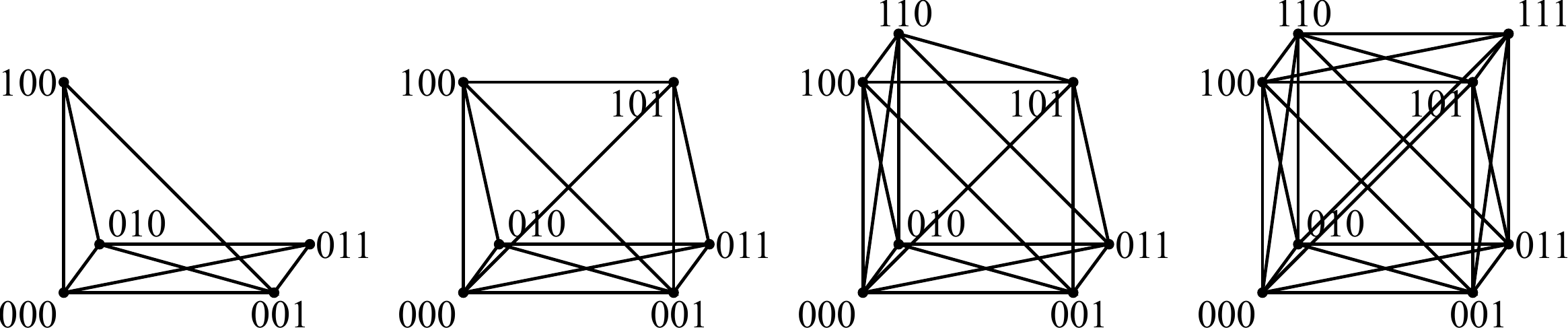}
\caption{
The graphs $G_m^2$ for $m=5,6,7,8$.
}
\label{fig:Gm2}
\end{figure}

Let $L_m^r$ be the subgraph of $G_m^r$ induced by the vertex set consisting of the (open) neighborhood of the ``last" vertex $m-1$.
That is, the vertices of $L_m^r$ are all vertices from $0$ up to $m-2$ whose Hamming distance from $m-1$ is at most $r$, and we have an edge between two vertices of $L_m^r$ if and only if their Hamming distance is at most $r$; see Figure~\ref{fig:Lm2}.
Note that $\cl(L_m^r)$ is the link of the vertex $m-1$ in the simplicial complex $\cl(G_m^r)$; this is the reason for our interest in the subgraphs $L_m^r$.

\begin{figure}[h]
\includegraphics[width=\textwidth]{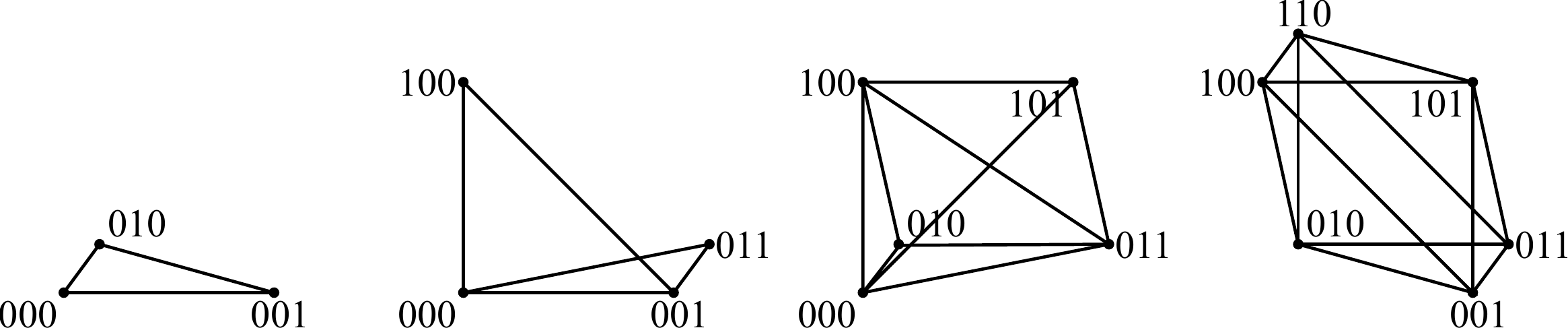}
\caption{
The graphs $L_m^2$ for $m=5,6,7,8$.
Note $\cl(L_8^2)=S^2$, as it is the boundary of the 3-dimensional cross-polytope.
}
\label{fig:Lm2}
\end{figure}

We now focus attention again on the case $r=2$.
Suppose we have a nonnegative integer $x$ with binary representation $x=2^{i_1}+2^{i_2}+\cdots+2^{i_l}$, where $i_1>i_2>\cdots >i_l$.
Then define
\begin{equation}
\label{eq:alpha}
\alpha(x)=(i_3+1)+2(i_4+1)+\cdots+(l-2)(i_l+1)=\sum_{s=3}^l (s-2)(i_s+1)=\sum_{s=3}^l (s-2)i_s + {\textstyle{l-1 \choose 2}}.
\end{equation}
In particular, note $\alpha(x)=0$ if and only if $x$ has at most two ones in its binary representation.
For example, we compute $\alpha(7)=1$.
The following lemma will generalize the fact that $\cl(L_8^2)\simeq\bigvee^{\alpha(7)}S^2 = S^2$.

\begin{lemma}\label{lem:link}
For all $m\ge 1$, we have a homotopy equivalence
\[\cl(L_m^2)\simeq\textstyle{\bigvee^{\alpha(m-1)}S^2}.\]
\end{lemma}

\begin{proof}
If $x$ is a nonnegative integer written in binary notation and $A\subseteq\N$ is a finite set, then we denote by $x_A$ the number obtained from $x$ by flipping the $i$-th digit for all $i\in A$.
We will write $x_i$ or $x_{i,j}$ as a shorthand for $x_{\{i\}}$ and $x_{\{i,j\}}$ for $i>j$, respectively.

Let $x=m-1$.
The vertices of $L_m^2$ are all numbers \emph{smaller than $x$} which can be obtained from $x$ by flipping one or two digits.
These are
\begin{itemize}
\item[(i)] $x_i$ whenever the $i$-th digit of $x$ is $1$, and
\item[(ii)] $x_{i,j}$, for all $i>j\geq 0$, whenever the $i$-th digit of $x$ is $1$.
\end{itemize}
There is an edge between vertices $x_A$ and $x_B$ in $L_m^2$ if and only if the symmetric difference satisfies $|A\triangle B|\leq 2$.

Recall that we decompose $x$ as $x=2^{i_1}+\cdots+2^{i_l}$, where $i_1>\cdots>i_l$.
In particular, all the vertices of the form (i) above are $x_{i_1},\ldots,x_{i_l}$, and the vertices for (ii) are $x_{i,j}$, where $i=i_k$ for some $k$, and where $0\leq j<i_k$ is arbitrary.
For $1\leq k\leq l+1$, let $L_{m,k}^2$ be the subgraph of $L_m^2$ induced by all the vertices of the form $x_{i,j}$,  together with $x_{i_k},\ldots,x_{i_l}$.
In particular $L_{m,1}^2=L_m^2$ is the full graph, and $L_{m,l+1}^2$ denotes the subgraph induced by the $x_{i,j}$ vertices only; see Figure~\ref{fig:L28k}.

\begin{figure}[h]
\includegraphics[width=\textwidth]{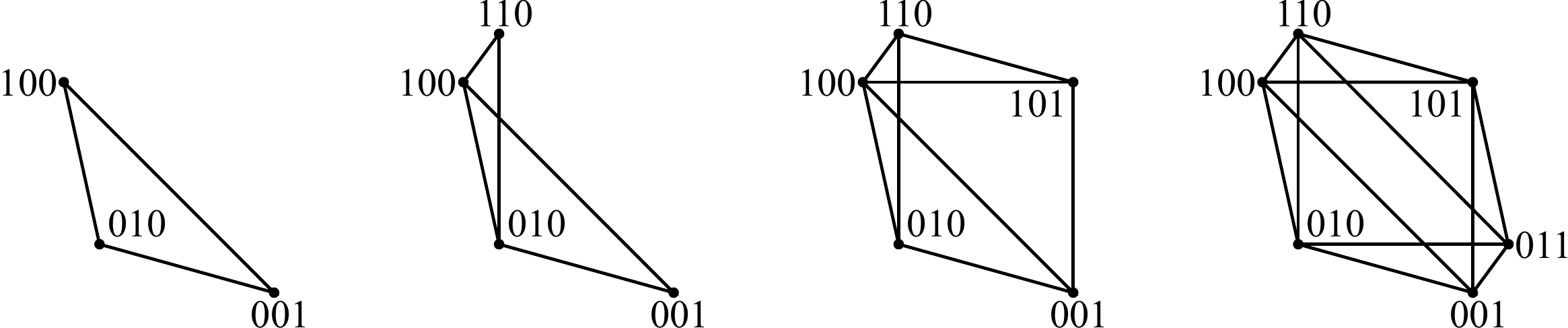}
\caption{
Let $m=8$, so $x=m-1=2^2+2^1+2^0$ is $111$ in binary.
We show the graphs graphs $L_{8,k}^2$ for $k=4,3,2,1$.
Note $L_{8,1}^2=L_8^2$.
}
\label{fig:L28k}
\end{figure}
 
We will prove by induction on $k$, starting from $k=l+1$ and going down to $k=1$, that $\cl(L_{m,k}^2)$ is homotopy equivalent to a wedge of 2-spheres, where the number of 2-spheres in the wedge sum is
\begin{equation}
\label{eq:num2spheres}
\sum_{s=3}^l (s-2)i_s + \sum_{s=k}^{l-2
}(l-s-1).
\end{equation}
The lemma will then follow by taking $k=1$, since $\sum_{s=1}^{l-2}(l-s-1)=\binom{l-1}{2}$, and so the number of 2-spheres is $\alpha(m-1)$ (see Equation~\eqref{eq:alpha}).

We start the induction at $k=l+1$.
Recall that $L_{m,l+1}^2$ is the subgraph of $L_m^2$ induced by the $x_{i,j}$ vertices only.
Note that if $|A|=|B|=2$, then $x_A$ and $x_B$ are connected by an edge in $L_{m,l+1}^2$ if and only if $A\cap B\neq \emptyset$.
This establishes $\cl(L_{m,l+1}^2)$ as an induced subcomplex\footnote{If $x=2^n-1$, i.e.\
if $m=2^n$, then we have exactly the independence complex of the Kneser graph, since there is a vertex $x_{A}$ for every subset $A\subseteq\{0,\ldots,n-1\}$ with $|A|=2$.
A sign that this method may be harder for $r>2$ is that the independence complexes of higher Kneser graphs are not well-understood, to our knowledge.} of the independence complex of a Kneser graph with vertices (such as $x_{i,j}$) corresponding to subsets of size two.
We will determine the homotopy type of $\cl(L_{m,l+1}^2)$ by borrowing Barmak's ideas from \cite[Theorem~4.11]{barmak2013star}.
The maximal simplices of $\cl(L_{m,l+1}^2)$ are of the form
\begin{itemize}
\item$\sigma(i_s):=\{x_{i_1,i_s},\ldots,x_{i_{s-1},i_s}\}\cup\{x_{i_s,j}~|~j<i_s\}$, as $s$ varies over $1,\ldots,l$, or
\item $\tau(i_t,i_s,j):=\{x_{i_t,i_s},x_{i_t,j},x_{i_s,j}\}$, for $i_t>i_s>j$.
\end{itemize}
The star cluster $\SC(\sigma(i_1))=\SC_{\cl(L_{m,l+1}^2)}(\sigma(i_1))$ contains all of the simplices $\sigma(i_s)$ because $\sigma(i_s)\in\st(x_{i_1,i_s})\subseteq \SC(\sigma(i_1))$ for every $s$.
Also, $\tau(i_1,i_s,j)\in \SC(\sigma(i_1))$ for any $i_1>i_s>j$.
However if $t\ge 2$, then the simplex $\tau(i_t,i_s,j)$ is not in $\SC(\sigma(i_1))$, although its boundary is.
Therefore, $\cl(L_{m,l+1}^2)$ is obtained from $\SC(\sigma(i_1))$ by attaching 2-cells along their boundaries, one 2-cell for each triple $i_t,i_s,j$ with $i_t>i_s>j\ge 0$ and $t\ge 2$.
Recall that $\SC(\sigma(i_1))$ is contractible by~\cite[Lemma~3.2]{barmak2013star}, and therefore $\cl(L_{m,l+1}^2)$ is homotopy equivalent to the quotient $\cl(L_{m,l+1}^2)/\SC(\sigma(i_1))$, which by the prior sentence is homotopy to a wedge sum of 2-spheres.
The number of 2-spheres is equal to the number of $3$-tuples $i_t>i_s>j\geq 0$ with $t\ge 2$, which is
$\sum_{s=3}^l (s-2)i_s$.
Indeed, once we fix $s$, then there are $i_s$ ways to choose $j$ and $s-2$ ways to choose $t$.
So $\cl(L_{m,l+1}^2)$ is homotopy equivalent to a wedge sum of $\sum_{s=3}^l (s-2)i_s$ copies of the 2-sphere, as desired for Equation~\eqref{eq:num2spheres}.


For the inductive step, suppose we have Equation~\eqref{eq:num2spheres} for $k+1\le l+1$, and we want to prove it for $k$.
Recall that $L_{m,k}^2$ arises from $L_{m,k+1}^2$ by adding the vertex $x_{i_k}$, i.e.\ $\cl(L_{m,k}^2)\setminus x_{i_k}=\cl(L_{m,k+1}^2)$.
Consider the link $\lk_{\cl(L_{m,k}^2)}(x_{i_k})$.
It is the clique complex of the subgraph of $L_m^2$ induced by the vertex set $X_1\cup X_2\cup X_3$, where we define
\begin{align*}
X_1&=\{x_{i_{k+1}},x_{i_{k+2}},\ldots,x_{i_l}\}\\
X_2&=\{x_{i_k,i_k-1},x_{i_k,i_k-2},\ldots,x_{i_k,0}\}\\
X_3&=\{x_{i_1,i_k},x_{i_2,i_k},\ldots,x_{i_{k-1},i_k}\}.
\end{align*}
Each of the subsets $X_1$ and $X_2\cup X_3$ induces a clique, hence a simplex in the clique complex.
(The subset $X_1$ is empty if $k=l$.)
Moreover, there are edges between $x_{i_s}$ and $x_{i_k,i_s}$ for every $s=k+1, \ldots, l$, and no other edges.
It follows that $\lk_{\cl(L_{m,k}^2)}(x_{i_k})$ is contractible when $k=l$, as it is a single simplex.
Similarly, $\lk_{\cl(L_{m,k}^2)}(x_{i_k})$ is contractible when $k=l-1$, as it is obtained by connecting two simplices via a single edge.
Furthermore, $\lk_{\cl(L_{m,k}^2)}(x_{i_k})$ is homotopy equivalent to a wedge of $l-k-1$ copies of the circle $S^1$ for $k\le l-2$.
The inclusion $\lk_{\cl(L_{m,k}^2)}(x_{i_k})\hookrightarrow \cl(L_{m,k+1}^2)$ is null-homotopic since, by induction, the latter is homotopy equivalent to a wedge of $2$-spheres.
By Lemma~\ref{lem:splitting} this leads to the splitting
\[\cl(L_{m,k}^2)\simeq \cl(L_{m,k+1}^2) \vee \Sigma\ \lk_{\cl(L_{m,k}^2)}(x_{i_k})\]
up to homotopy type.
The induction step proving Equation~\eqref{eq:num2spheres} is therefore complete, since the number of $S^2$ wedge summands increases by zero if $k=l$ or $l-1$, and by $l-k-1$ if $1\leq k\le l-2$.
\end{proof}

\begin{figure}[h]
\captionsetup{width=0.96\textwidth}
\centering
\includegraphics[width=1.5in]{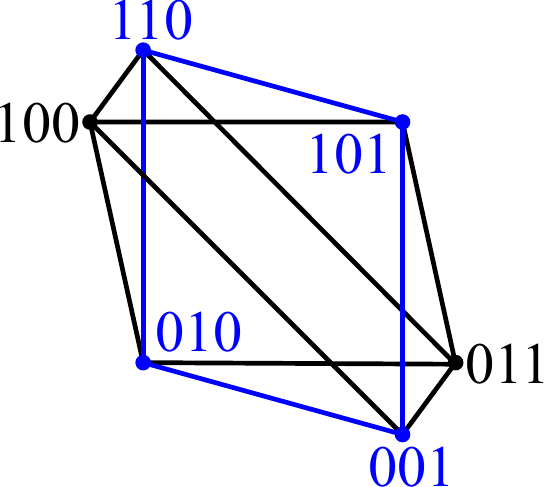}
\caption{
A figure for the inductive step for Equation~\eqref{eq:num2spheres} in the proof of Lemma~\ref{lem:link}.
Let $m=8$, so $x=m-1=2^2+2^1+2^0$ is $111$ in binary.
We have $i_1=2$, and $x_{i_1}=2^1+2^0$ is $011$ in binary.
We see that $\lk_{\cl(L_{8,1}^2)}(x_{i_1})$, drawn in blue, is a single circle.
In this example, $X_1$ is the single edge between $101$ and $110$, and $X_2\cup X_3=X_2$ is the single edge between $001$ and $010$.}
\end{figure}

\begin{theorem}\label{thm:Gm2}
The space $\cl(G_m^2)$ is homotopy equivalent to a wedge of $\sum_{k=0}^{m-1}\alpha(k)$ copies of the $3$-sphere.
\end{theorem}

\begin{proof}
We proceed by induction on $m$.
For base case $m=1$, note that $\cl(G_1^2)$ is a single vertex.
This agrees with the fact that for $m=1$ the sum returns $\alpha(0)=0$, and the 0-fold wedge sum of 3-spheres is a single point as well.

For the inductive step, assume that $\cl(G_{m-1}^2)$ is homotopy equivalent to a wedge of $\sum_{k=0}^{m-2}\alpha(k)$ copies of the 3-sphere.
Recall that $\cl(G_m^2)\setminus m = \cl(G_{m-1}^2)$, and that $\lk_{\cl(G_m^2)}(m) = \cl(L_m^2)$.
We now consider the inclusion $\cl(L_m^2)\hookrightarrow \cl(G_{m-1}^2)$.
By induction, and by the description of $\cl(L_m^2)$, this map is null-homotopic, because it goes from a homotopy wedge of 2-spheres to a homotopy wedge of 3-spheres.
Therefore, by Lemma~\ref{lem:splitting} we have
\[\cl(G_m^2)\simeq \cl(G_{m-1}^2)\vee \Sigma\ \cl(L_m^2)\]
up to homotopy type.
By Lemma~\ref{lem:link} this proves that $\cl(G_m^2)$ is homotopy equivalent to a wedge of 3-spheres, and that the number of 3-spheres has increased by $\alpha(m-1)$ over those in $\cl(G_{m-1}^2)$.
This give $\sum_{k=0}^{m-1}\alpha(k)$ copies of the 3-sphere in total.
\end{proof}

The main theorem now follows as a special case.

\begin{proof}[Proof of Theorem~\ref{thm:main}]
Let $n\ge 3$.
Our task is to show that the Vietoris--Rips complex $\vr{Q_n}{2}$ is homotopy equivalent to the $c_n$-fold wedge sum of 3-spheres, namely 
\[\vr{Q_n}{2}\simeq\textstyle{\bigvee^{c_n}}S^3,\quad\quad\text{where}\quad\quad c_n=\displaystyle{\sum_{0\le j<i<n}(j+1)(2^{n-2}-2^{i-1})}.\]
From Theorem~\ref{thm:Gm2}, we have that
\[ \vr{Q_n}{2}=\cl(G_{2^n}^2)\simeq\textstyle{\bigvee^{\sum_{k=0}^{2^n-1}\alpha(k)}} S^3. \]
We compute $\sum_{k=0}^{2^n-1}\alpha(k)$ as follows.
Recall $\alpha(k)=(i_3+1)+2(i_4+1)+\cdots+(l-2)(i_l+1)$ if $k=2^{i_1}+2^{i_2}+\cdots+2^{i_l}$ with $i_1>i_2>\cdots >i_l$.
By looking carefully, we see that the sum $\sum_{k=0}^{2^n-1}\alpha(k)$ can be obtained as follows.
For every pair of positions $(i,j)$ with $i>j$ in the binary representation, if we add $j+1$ once for every instance where $k=\cdots+2^i+\cdots+2^j+\cdots\le 2^n-1$ and $2^i$ is not the leading term in $k$, then the result will be $\sum_{k=0}^{2^n-1}\alpha(k)$.
There are $2^{n-2}-2^{i-1}$ numbers $k$ with $0\le k\le 2^n-1$ for which that happens, giving
\[\sum_{k=0}^{2^n-1}\alpha(k) = \displaystyle{\sum_{0\le j<i<n}(j+1)(2^{n-2}-2^{i-1})} = c_n.\]
\end{proof}

\section{Conclusion and open questions}\label{sec:conclusion}

We conclude with a description of some open questions.
We remind the reader of Questions~\ref{ques:wedge}--\ref{ques:homotopy} in the introduction, which ask if $\vr{Q_n}{r}$ is always a wedge of spheres, and what the homology groups and homotopy types of $\vr{Q_n}{r}$ are for $3\le r\le n-2$.
Below we give some more refined versions of these questions.

\begin{question}
Does $\cl(G_m^r)$ collapse to its $(2^r-1)$-skeleton?
If so, this would in particular imply that the homology groups $H_i(\cl(G_m^r))$ are zero for $i\geq 2^r$.
\end{question}

\begin{question}\label{ques:homology-summand}
Is $\redhom_i(\cl(G_m^r))\cong\redhom_i(\cl(G_{m-1}^r))\oplus \redhom_{i-1}(\cl(L_m^r))$, where $\redhom_i$ denotes reduced homology?
\end{question}

\begin{question}\label{ques:homotopy-summand}
Is the inclusion $\iota\colon \cl(L_m^r)\hookrightarrow \cl(G_{m-1}^r)$ null-homotopic for all $m$ and $r$?
\end{question}

Recall that $\cl(L_m^r)$ is the link of the vertex $m-1$ in $\cl(G_m^r)$.
The inclusion $\iota\colon\cl(L_m^r)\hookrightarrow \cl(G_{m-1}^r)$ has as cofibre $\cl(G_m^r)$, i.e.\
\[\cl(G_m^r)=\cl(G_{m-1}^r)\cup_\iota \st_{\cl(G_m^r)}(m-1).\]
By Lemma~\ref{lem:splitting}, an affirmative answer to Question~\ref{ques:homotopy-summand} would provide a homotopy equivalence $\cl(G_m^r)\simeq \cl(G_{m-1}^r)\vee \Sigma\ \cl(L_m^r)$.
So an affirmative answer to Question~\ref{ques:homotopy-summand} would also provide an affirmative answer to Question~\ref{ques:homology-summand}, and it would also follow that the inclusion $\cl(G_{m-1}^r)\hookrightarrow \cl(G_m^r)$ induces an injection in homology.

In the proof of Theorem~\ref{thm:Gm2} we show that Question~\ref{ques:homotopy-summand} is answered in the affirmative when $r=2$.
Computational evidence shows that Question~\ref{ques:homology-summand} is also answered in the affirmative when $r=3$ and $n\leq 64$. 

\section{Acknowledgements}

We would like to thank Kevin Emmett, Ra\'{u}l Rabad\'{a}n, and Daniel Rosenbloom for raising the question of Vietoris--Rips complexes of hypercube graphs.
We would like to thank Karim Adiprasito, Florian Frick, and Simon Zhang for helpful conversations.

\bibliographystyle{plain}
\bibliography{OnVietorisRipsComplexesOfHypercubeGraphs.bib}





\end{document}